\documentclass[smallcondensed]{svjour3}

\smartqed  
\usepackage{amsmath,amssymb,amsfonts}

\begin{document}

\title{Full description of the eigenvalue set of the $(p,q)$-Laplacian with a Steklov-like boundary condition
}

\titlerunning{Full description of the eigenvalue set of the $(p,q)$-Laplacian}

\author{Lumini\c{t}a Barbu   \and
        Gheorghe Moro\c{s}anu
}


\institute{Lumini\c{t}a Barbu \at
               Faculty of Mathematics and Computer Science \\
               Ovidius University\\
                124 Mamaia Blvd, 900527 Constan\c{t}a, Romania \\
                           \email{lbarbu@univ-ovidius.ro}           
           \and
          Gheorghe Moro\c{s}anu \at
      Academy of Romanian Scientists, Bucharest, Romania
             \\
             and \\
              Faculty of Mathematics and Computer Science\\
              Babe\c{s}-Bolyai University\\
              1 M. Kog\u{a}lniceanu Str., 400084 Cluj-Napoca, Romania\\
               \email{morosanu@math.ubbcluj.ro}
               }

\date{Received: date / Accepted: date}

\maketitle

\begin{abstract}
In this paper we consider in a bounded domain $\Omega \subset \mathbb{R}^N$ with smooth boundary an eigenvalue problem for the negative $(p,q)$-Laplacian
with a Steklov-like boundary condition, where $p,\, q\in (1,\infty)$, $p\neq q$, including the open case $p\in (1,\infty)$, $q\in (1, 2)$, $p\neq q$. A full description of the set of eigenvalues of this problem is provided. Our results complement those previously obtained by Abreu and Madeira \cite{AM}, Barbu and Moro\c{s}anu \cite{BM}, F\u{a}rc\u{a}\c{s}eanu, Mih\u{a}ilescu and Stancu-Dumitru \cite{FMS}, Mih\u{a}ilescu \cite{MMih}, Mih\u{a}ilescu and Moro\c{s}anu \cite{MM}.
\end{abstract}

\section{ Introduction}
Let $\Omega \subset \mathbb{R}^N$ be a bounded domain with smooth  boundary $\partial\Omega$. Consider in $\Omega$ the eigenvalue problem
\begin{equation}\label{eq:1.1}
\left\{\begin{array}{l}
\mathcal{A} u:=-\Delta_p u-\Delta_q u=\lambda a(x) \mid u\mid ^{q-2}u\ \ \mbox{ in} ~ \Omega,\\[1mm]
\frac{\partial u}{\partial\nu_\mathcal{A}}=\lambda b(x) \mid u\mid ^{q-2}u ~ \mbox{ on} ~ \partial \Omega,
\end{array}\right.
\end{equation}
under the following hypotheses

$(h_{pq}) \ \ \ \ \ \ p,~q\in (1, \infty),~ p\neq q$;\\

$(h_{ab}) \ \ \ \ \ \ a\in L^{\infty}(\Omega)$ and $b\in L^{\infty}(\partial \Omega)$ are given nonnegative functions satisfying
\begin{equation}\label{eq:1.2}
 \int_\Omega a(x)~dx+\int_{\partial\Omega} b(\sigma)~d\sigma >0.
\end{equation}
We have used above the notation
$$
\frac{\partial u}{\partial\nu_\mathcal{A}}:=\big(\mid \nabla u\mid ^{p-2}+\mid \nabla u\mid ^{q-2}\big)\frac{\partial u}{\partial\nu},
$$
where $\nu$ is the unit outward normal to $\partial\Omega$. As usual, $\Delta_p$ denotes the $p$-Laplacian, i.e.,
$\Delta_pu=\, \mbox{div} \, (|\nabla u|^{p-2}\nabla u)$.

The operator $\big(\Delta_p + \Delta_q\big)$,
called $(p,q)$-Laplacian, occurs in many applications
in physics and related sciences such as biophysics (see \cite{F}, \cite{Mu}), quantum and
plasma physics (see \cite{A}, \cite{W}), solid state physics (\cite{My}),  chemical reaction design (see \cite{Ar}), etc.

\bigskip

The solution $u$ of \eqref{eq:1.1} is understood in a weak sense, as an element of the Sobolev space $W:=W^{1,\max\{p,q\}}(\Omega)$ satisfying
equation $\eqref{eq:1.1}_1$ in the sense of distributions and $\eqref{eq:1.1}_2$ in the sense of traces.
\begin{definition}\label{def1}
$\lambda\in \mathbb{R}$ is an eigenvalue of problem \eqref{eq:1.1} if there exists $u_\lambda\in W \setminus \{0\}$ such that
\begin{equation}\label{eq:1.3}
\begin{split}
\int_\Omega &\Big(\mid \nabla u_\lambda\mid ^{p-2}+\mid \nabla u_\lambda\mid ^{q-2}\Big)\nabla u_\lambda \cdot \nabla w~dx \\
&=\lambda\Big(\int_\Omega a\mid  u_\lambda\mid ^{q-2} u_\lambda  w~dx+\int_{\partial\Omega} b \mid  u_\lambda\mid ^{q-2} u_\lambda  w~d\sigma\Big)~\forall~w\in W.
\end{split}
\end{equation}
\end{definition}
According to a Green type formula (see \cite{CF}, p. 71), $u\in W$ is a solution of \eqref{eq:1.1} if and only if it satisfies \eqref{eq:1.3}.

 Choosing $w=u_\lambda$ in \eqref{eq:1.3} shows that the eigenvalues of problem \eqref{eq:1.1} cannot be negative.
It is also obvious that $\lambda_0=0$ is an eigenvalue of this problem and the corresponding eigenfunctions
are the nonzero constant functions. So any other eigenvalue belongs to $(0,\infty)$.

If we assume that $\lambda>0$ is an eigenvalue of problem \eqref{eq:1.1} and choose $w\equiv 1$ in \eqref{eq:1.3} we deduce that every eigenfunction
$u_{\lambda}$ corresponding to $\lambda$ satisfies the equation
\begin{equation}\label{eq:1.4}
\int_\Omega a\mid  u_\lambda\mid ^{q-2} u_\lambda  ~dx+\int_{\partial\Omega} b \mid  u_\lambda\mid ^{q-2} u_\lambda  ~d\sigma=0.
\end{equation}
So all eigenfunctions corresponding to positive eigenvalues necessarily belong to the set
\begin{equation}\label{eq:1.5}
\mathcal{C}:=\Big\{ u\in W;~\int_\Omega a\mid  u\mid ^{q-2} u  ~dx+\int_{\partial\Omega} b \mid  u\mid ^{q-2} u  ~d\sigma=0\Big\}.
\end{equation}
This set is a symmetric cone and for $q=2$ it is a linear subspace of $W$.

In the particular case $q=2,~a\equiv 1,~b\equiv 0$, the set of eigenvalues for problem \eqref{eq:1.1} was completely described by
M. Mih\u{a}ilescu \cite{MMih} (for $p>2)$ and  M. F\u{a}rc\u{a}\c{s}eanu, M. Mih\u{a}ilescu and D. Stancu-Dumitru \cite{FMS} (for $p\in (1, 2)$). Problem \eqref{eq:1.1}  with $q=2, ~p\in (1,\infty)\setminus\{2\}$ has been studied by J. Abreu and G. Madeira \cite{AM}. We point out that in the case $q=2$ the techniques employed in the papers just mentioned are not applicable to the case of the $(p, q)$-Laplacian with $q\neq 2 $  since in this situation  $\mathcal{C}$  is no longer a linear subspace of $W$.
Note that  problem \eqref{eq:1.1} with $p\in (1, \infty), \ q\in (2, \infty), \ p\neq q, \ a\equiv 1, \ b\equiv 0$ has been investigated  by M. Mih\u{a}ilescu and G. Moro\c{s}anu in \cite{MM}; also,  problem \eqref{eq:1.1} with $p\in (1,\infty)$, $q\in (2,\infty)$ \ $p\neq q$ has been solved by L. Barbu and G. Moro\c{s}anu \cite{BM}. The strategy employed in these two papers, based on the Lagrange Multipliers Rule, cannot be applied to the case $p\in (1,\infty )$,  $q\in (1,2)$, $p\neq q$, since the constraint set $\mathcal{C}$ defined in \eqref{eq:1.5} is no longer a $C^1$ manifold. This case requires separate analysis and some difficulties that occur within the new framework have to be overcome. We shall make use of the so-called direct methods in the Calculus of Variations. In fact, the arguments we shall use work for all $q\in (1, \infty),$ not just for $q\in (1,2).$

\bigskip

Specifically, our goal here is to determine the set of all eigenvalues of problem \eqref{eq:1.1} under $(h_{pq})$ and $(h_{ab})$. As we have already mentioned, in  \cite[Theorem 1.1]{AM} and \cite[Theorem 3.1]{BM} it was proved that in the cases $p\in (1,\infty)\setminus \{ 2\}$, $q=2$ and $p\in (1,\infty)$, $q\in (2, \infty)$, $p\neq q$, respectively, the set of eigenvalues of problem \eqref{eq:1.1} is given by $0\cup (\lambda_1, \infty),$ where $\lambda_1$ is given by
\begin{equation}\label{eq:1.6}
\lambda_1:=\underset{w\in\mathcal{C}\setminus\{0\}}{\inf }~\frac{\int_\Omega\mid\nabla w\mid^q~dx}{\int_{\Omega}a\mid w\mid^q~dx+
\int_{\partial\Omega}b\mid  w\mid^q~d\sigma}.
\end{equation}
Note that the denominators of the above fractions may equal zero for some $w$'s in $\mathcal{C}\setminus \{ 0\}$ and in such cases the corresponding
numerators are obviously $> 0,$ thus the values of those fractions are considered $\infty$ so they do not contribute to $\lambda_1$.

\bigskip

Let us now state the main result of this paper (which covers the open case $p\in (1,\infty), \ q\in (1,2), \ p\neq q$).
\begin{theorem}\label{teorema1}
Assume that $(h_{pq})$ and $(h_{ab})$ above are fulfilled. Then the set of eigenvalues of problem \eqref{eq:1.1} is precisely $\{0\} \cup (\lambda_1, \infty)$, where $\lambda_1$ is the positive constant defined by \eqref{eq:1.6}.
\end{theorem}

The conclusion that the eigenvalue set contains an interval is due to the fact that the operator $\mathcal{A}$ is
nonhomogeneous ($p\neq q$). Note also that Theorem \ref{teorema1} provides a full description of the eigenvalue set of $\mathcal{A}$.

On the other hand, a complete description of the eigenvalue set in the homogeneous case $p=q$ is not known even in particular cases.  For example, if $p=q>1$, $a\equiv 1$, $b\equiv 0$, then the eigenvalue set of the corresponding problem is fully known only if $p=q=2$ (i.e., $\mathcal{A}=-2\Delta$); otherwise, i.e. if $p=q\in (1, \infty)\setminus \{ 2\}$, then it is only known, as a consequence of the
Ljusternik-Schnirelman theory, that there exists a sequence of positive eigenvalues of problem \eqref{eq:1.1} with $\mathcal{A}=-2\Delta_p$ (see, e.g., \cite[Chap. 6]{GP}), but this sequence may not constitute the whole eigenvalue set.

\section{Preliminary results}

Let $q\in (1,\infty)$ be arbitrary but fixed. As we have pointed out in Introduction, all eigenfunctions corresponding to positive eigenvalues necessarily belong to the set
\[
\mathcal{C}:=\Big\{ u\in W;~\int_\Omega a\mid  u\mid ^{q-2} u  ~dx+\int_{\partial\Omega} b \mid  u\mid ^{q-2} u  ~d\sigma=0\Big\}.
\]
This is a symmetric cone. Moreover,   $\mathcal{C}$ is a weakly closed subset of $W$. Indeed, let $\big(u_n\big)_n\subset \mathcal{C}$ such that $u_n\rightharpoonup u_0$ in $W.$ Since  $W\hookrightarrow L^q (\Omega)$ and $W\hookrightarrow L^q (\partial\Omega)$ compactly,  there exists a subsequence of $\big(u_n\big)_n$, also denoted $\big(u_n\big)_n$, such that
\[
u_n\rightarrow u_0 ~\mbox{in} ~L^{q}(\Omega),~u_n\rightarrow u_0 ~\mbox{in} ~L^{q}(\partial\Omega).
\]
By Lebesgue's Dominated Convergence Theorem (see also \cite[Theorem 4.9] {Br}) we obtain $u_0\in \mathcal{C}.$ In addition, $\mathcal{C}$ has nonzero elements (see \cite[Section 2]{BM}).

Now, for $r>1$ define the set
\[
\mathcal{C}_r:=\Big\{ u\in W^{1,r}(\Omega);~\int_\Omega a\mid  u\mid ^{r-2} u  ~dx+\int_{\partial\Omega} b \mid  u\mid ^{r-2} u  ~d\sigma=0\Big\}.
\]
Arguing as before, we infer that for all $r>1,$ $\mathcal{C}_r$ is a symmetric, weakly closed (in $W^{1,r}(\Omega)$) cone, containing infinitely many nonzero elements.
Note also that $\mathcal{C}=\mathcal{C}_q$ if $q>p$, otherwise (i.e., if $q<p$), then  $\mathcal{C}$ is a proper subset of $\mathcal{C}_q.$

Next, for $q>1,$ we consider the eigenvalue problem
\begin{equation}\label{eq:2.1}
\left\{\begin{array}{l}
-\Delta_q u=\lambda a(x) \mid u\mid ^{q-2}u\ \ \mbox{in} ~ \Omega,\\[1mm]
\mid \nabla u\mid^{q-2}\frac{\partial u}{\partial\nu}=\lambda b(x) \mid u\mid ^{q-2}u  ~ \mbox{on} ~ \partial\Omega.
\end{array}\right.
\end{equation}

As usual, the number $\lambda\in \mathbb{R}$ is said to be an eigenvalue of problem \eqref{eq:2.1} if there exists a function $u_\lambda\in W^{1,q} (\Omega\setminus \{0\})$ such that
\begin{equation}\label{eq:2.2}
\begin{split}
\int_\Omega &\mid \nabla u_\lambda\mid ^{q-2}\nabla u_\lambda \cdot \nabla w~dx \\
=&\lambda\Big(\int_\Omega a\mid  u_\lambda\mid ^{q-2} u_\lambda  w~dx+\int_{\partial\Omega} b \mid  u_\lambda\mid ^{q-2} u_\lambda  w~d\sigma\Big)~\forall~w\in W^{1,q}(\Omega).
\end{split}
\end{equation}
Obviously, $\lambda_0=0$ is an eigenvalue of problem \eqref{eq:2.1} and any other eigenvalue belongs to $(0, \infty).$ Moreover, if we consider an eigenvalue $\lambda >0$ of \eqref{eq:2.1} and choose $w\equiv 1$ in \eqref{eq:2.2}, we deduce that every eigenfunction $u_\lambda$ corresponding to $\lambda$ belongs to $\mathcal{C}_{q}\setminus\{0\}.$
We also define
\begin{equation}\label{eq:2.3}
\lambda_{1q}:=\underset{w\in\mathcal{C}_{1q}\setminus\{0\}}{\inf }~\frac{\int_\Omega \mid\nabla w\mid^q~dx}{\int_{\Omega} a\mid  w\mid^q~dx+\int_{\partial\Omega} b\mid\nabla w\mid^q~d\sigma}.
\end{equation}

Now, let us consider the functional
$$J: W^{1,q}(\Omega)\rightarrow \mathbb{R},~J(w):=\int_\Omega \mid  \nabla w\mid ^{q} ~dx,
$$
which is positively homogeneous of order $q$. By standard arguments we can infer that functional $J$ is  convex and weakly lower semicontinuous for all $q>1.$

Consider the minimization problem
\begin{equation}\label{eq:2.4}
\underset{w\in\mathcal{C}_{1q}}{\inf }~J(w) \, ,
\end{equation}
where \[
\mathcal{C}_{1q}:=\mathcal{C}_{q}\cap\Big\{ u\in W^{1,q}(\Omega);\int_\Omega a\mid  u\mid ^{q} ~dx+\int_{\partial\Omega} b\mid u\mid^q~d\sigma=1\Big\}.
\] The next result states that $J$ attains its minimal value over the set $\mathcal{C}_{1q},$  this value is positive and is equal to $\lambda_{1q}$.

\begin{lemma}\label{lema1}
If $q\in (1, \infty),$ then there exists $u_{*}\in \mathcal{C}_{1q}$ such that $$\mu:=J(u^{*})=\underset{w\in\mathcal{C}_{1q}}{\inf }~J(w)>0.$$
Moreover, $\mu=\lambda_{1q}$ and it is the lowest positive eigenvalue of problem \eqref{eq:2.1} with  eigenfunction $u_*$.
\end{lemma}
\begin{proof}
It is well-known that functional $J$ is of class $C^1$ on $W^{1,q}(\Omega)$ and obviously $J$ is bounded below. Let $\big(u_n\big)_n\subset \mathcal{C}_{1q}$
be a minimizing sequence for $J$, i. e.,
$$
J(u_n)\rightarrow \underset{w\in\mathcal{C}_{1q}}{\inf }~J(w)=\mu.
$$
Let us prove that $\big(u_n\big)_n$ is bounded in $W^{1,q}(\Omega)$. Assume the contrary, that there exists a subsequence of $\big(u_n\big)_n$, again
denoted $\big(u_n\big)_n$, such that $\parallel u_n\parallel_{W^{1,q}(\Omega)}\rightarrow\infty$ as $n\rightarrow\infty.$ Define
$$
v_n:=\frac{u_n}{\parallel u_n\parallel_{W^{1,q}(\Omega)}} \ \ \ \forall~n\in \mathbb{N}\, .
$$
Clearly, the sequence $\big(v_n\big)_n$ is bounded in $ W^{1,q}(\Omega)$ so there exist a $v\in W^{1,q}(\Omega)$ and a subsequence of $\big(v_n\big)_n$, again denoted
$\big(v_n\big)_n$, such that
\[
v_n\rightharpoonup v ~\mbox{in} ~W^{1,q}(\Omega).
\]
Since $W^{1,q}(\Omega)\hookrightarrow L^q (\Omega)$ and $W^{1,q}(\Omega)\hookrightarrow L^q (\partial\Omega)$ compactly,  we have up to a subsequence
\[
v_n\rightarrow v ~\mbox{in} ~L^{q}(\Omega),~v_n\rightarrow v ~\mbox{in} ~L^{q}(\partial\Omega).
\]
As $\parallel v_n\parallel_{W^{1,q}(\Omega)}=1~\forall ~n\in \mathbb{N}$, we have $\parallel v\parallel_{W^{1,q}(\Omega)}=1$, and
\begin{equation*}
\begin{split}
\int_\Omega \mid \nabla v\mid ^{q}~dx\leq \underset{n\rightarrow \infty}{\liminf}~\int_\Omega \mid \nabla v_n\mid ^{q}~dx=\underset{n\rightarrow \infty}{\liminf}\frac{1}{\parallel u_n\parallel^q_{L^q(\Omega)}}J(u_n)=0,
\end{split}
\end{equation*}
which shows that $v$ is a constant function. On the other hand, since $\big(v_n\big)_n\subset \mathcal{C}_q$ and $\mathcal{C}_q$ is weakly closed in
$W^{1,q}(\Omega)$, we infer that $v\in \mathcal{C}_q$, hence $v\equiv 0$. But this contradicts the fact that $\parallel v\parallel_{W^{1,q}(\Omega)}=1$.
Therefore, $\big(u_n\big)_n$ is indeed bounded in $W^{1,q}(\Omega)$, hence there exist $u_*\in W^{1,q}(\Omega)$ and a subsequence of $\big(u_n\big)_n$, also
denoted $\big(u_n\big)_n$, such that
\[
u_n\rightharpoonup u_* ~\mbox{in} ~W^{1,q}(\Omega),
\]
\[
u_n\rightarrow u_* ~\mbox{in} ~L^{q}(\Omega),~u_n\rightarrow u_* ~\mbox{in} ~L^{q}(\partial\Omega).
\]
By Lebesgue's Dominated Convergence Theorem we obtain $u_*\in \mathcal{C}_{1q}$, so the weak lower semicontinuity of $J$ leads to $\mu=J(u_*).$
In addition, $J(u_*)>0$. Indeed, assuming by contradiction that $J(u_*)=0$ would imply that $u_* \equiv Const.$, which is impossible because $u_*\in
\mathcal{C}_{1q}$ (see also assumption $(h)_{ab}$).

Since the functional $J$ is positively homogeneous of order $q$, we have
\begin{equation}\label{eq:2.5}
\mu=\underset{w\in\mathcal{C}_q\setminus\{0\}}{\inf }~\frac{\int_\Omega\mid\nabla w\mid^q~dx}{\int_{\Omega}a \mid w\mid^q~dx+\int_{\partial\Omega} b \mid w\mid^q~d\sigma},
\end{equation}
thus, we derive from \eqref{eq:2.3} $\mu=\lambda_{1q}.$

We are now going to prove that $\mu=\lambda_{1q}$ is the lowest positive eigenvalue of problem \eqref{eq:2.1} with corresponding eigenfunction $u_*$.

\noindent For $q\in [2, \infty )$ the result has been proved in \cite[Remark 3.2]{BM}. If $q\in (1, 2)$, since the constraint set $\mathcal{C}_{q}$ is no longer a $C^1$ manifold,  we cannot use the Lagrange Multipliers Theorem as in \cite{BM}.

In order to overcome this inconvenience, let us define $J_\mu:W^{1,q}(\Omega)\rightarrow\mathbb{R},$
\begin{equation}\label{eq:2.6}
J_\mu(u)=\frac{1}{q}\int_\Omega \mid \nabla u\mid^{q}~dx-\frac{\mu}{q}\Big(\int_\Omega a \mid u\mid^q ~dx+\int_{\partial\Omega} b \mid u\mid^q ~d\sigma \Big)~\forall~u\in W^{1,q}(\Omega),
\end{equation}
which is a $C^1$ functional whose derivative is given by
\begin{equation}\label{eq:2.7}
\begin{split}
\langle{J}_\mu'(u),w\rangle&=
\int_\Omega\mid \nabla u\mid ^{q-2}\nabla u\cdot\nabla w~dx\\
&-\mu \Big( \int_\Omega a\mid u\mid^{q-2} uw~dx+\int_{\partial\Omega} b \mid u\mid^{q-2} u w ~d\sigma\Big)
\end{split}
\end{equation}
for all $u, w\in W^{1,q}(\Omega).$ In order to prove that $\mu=\lambda_{1q}$ is an eigenvalue of problem \eqref{eq:2.1} with eigenfunction $u_* \not\equiv 0$, it is sufficient to show that $J^{\prime}_{\mu}(u_*)=0.$ In this case we make use of an argument in \cite[Lemma 5.8]{BF}.

In this respect, we fix $v\in \mbox{Lip}(\Omega)$ arbitrarily and try to construct a sequence $\big(u_n\big)_n\subset \mathcal{C}_q$ such that $u_n\rightarrow u_*$ in $W^{1,q}(\Omega)$ as $n\rightarrow\infty.$ To this aim, let us define $\mathcal{I}:W^{1,q}(\Omega)\rightarrow \mathbb{R},$
\[
\mathcal{I}(w):=\int_{\Omega}a \mid w\mid^{q-2} w~dx+\int_{\partial\Omega} b \mid w\mid^{q-2} w~d\sigma~\forall~w\in W^{1,q}(\Omega),
\]
and for each $n\in \mathbb{N}^*,$
\begin{equation}\label{eq:2.8}
g_n:\mathbb{R}\rightarrow \mathbb{R}, ~~g_n(s):=\mathcal{I}\Bigl(u_*+\frac{1}{n}v+s\Bigr)~\forall~s\in \mathbb{R}.
\end{equation}
Since the function $s\rightarrow \mid w+s\mid ^{q-2} \big(w+s\big)$ is strictly increasing on $\mathbb{R},$  $g_n$ is increasing on $\mathbb{R}.$ In fact, $g_n$  is strictly increasing on $\mathbb{R}$ since, by virtue of  $(h_{a b})$, we see that \eqref{eq:1.2}
implies that either $|\{x\in \Omega;~a(x)> 0\}|_N>0$ or $a=0$ a.e. in $\Omega$ and $|\{x\in \partial\Omega;~b(x)> 0\}|_{N-1}>0$. Here $|\cdot |_N$, $|\cdot |_{N-1}$ denote the Lebesgue measures of the corresponding sets.

In order to show that for all $n\in \mathbb{N}^*$ there exists $s_n\in\mathbb{R}$ such that $g_n(s_n)=0$, i.e. $u_*+\frac{1}{n}v+s_n\in \mathcal{C}_q$, we also define $h_n:\mathbb{R}\rightarrow \mathbb{R},$
\begin{equation}\label{eq:2.9}
h_n(s)=\int_\Omega a\Bigl\lvert u_*+\frac{1}{n} v+s \Bigr\lvert^q ~dx+\int_{\partial\Omega} b \Bigl\lvert u_*+\frac{1}{n} v+s \Bigr\lvert^q ~d\sigma~ \forall~n\in \mathbb{N}^*~\forall~s\in\mathbb{R}.
\end{equation}
It is easily seen that $h_n$ is coercive, because
\begin{equation*}
\begin{split}
h_n(s)&\geq 2^{-q} \mid s\mid ^q \big( \parallel a\parallel_{L^\infty(\Omega)}\mid \Omega\mid_N +\parallel b\parallel_{L^\infty(\partial\Omega)}\mid \partial\Omega\mid_{N-1}\big)\\
&-\int_\Omega a\Bigl\lvert u_*+\frac{1}{n} v \Bigr\lvert^q ~dx-\int_{\partial\Omega}b \Bigl\lvert u_*+\frac{1}{n} v \Bigr\lvert^q ~d\sigma .
\end{split}
\end{equation*}
Here, we have also used the inequality $$\mid x\mid^q \leq (\mid x+y\mid+\mid y\mid)^q\leq 2^q(\mid x+y\mid^q+\mid y\mid^q)~\forall~x, y\in \mathbb{R},~q>1.$$
 Moreover, $h_n$ is continuously differentiable, $h_n^\prime=g_n$  (see \cite [Theorem 2.27]{Fo}) and convex (its derivative $g_n$ is an increasing function).
Therefore, for all $n\in \mathbb{N}^*,$ $h_n$ has a minimizer $s_n$, such that $h_n^\prime(s_n)=g_n (s_n)=0$.

Next, we want to show that the sequence $\big( n s_n\big)_n$ is bounded. Arguing by contradiction, let us assume that, after passing to a subsequence if necessary, $n s_n\rightarrow \infty$ or  $n s_n\rightarrow -\infty$ as $n\rightarrow\infty.$ Since $v\in \mbox{Lip} (\Omega),$  there exists $N_1$ large enough such that, we have either $v(\cdot)+n s_n>0 ~\mbox{in}~ \Omega,$ or $v(\cdot)+n s_n<0 ~\mbox{in}~ \Omega~\forall~n\geq N_1.$

Set
\begin{equation}\label{eq:2.10}
u_n:=u_*+\frac{1}{n} v+s_n~\forall~n\in \mathbb{N}^*.
\end{equation}
Obviously, $\big(u_n\big)_n \subset \mathcal{C}_q.$

Since the functions $g_n, \, n\geq N_1,$ are strictly increasing on $\mathbb{R}$, we have
\begin{equation}\label{eq:2.11}
0=g_n(u_n)> g(u_*)=0~~\forall~n\geq N_1,
\end{equation}
if $v(\cdot)+n s_n>0 ~\mbox{in}~ \Omega,$ or the reverse inequality in the latter case, when $v(\cdot)+n s_n<0 ~\mbox{in}~ \Omega $. So, in both cases we get a contradiction.

Consequently, the sequence $\big(n s_n\big)_n$ is indeed bounded. This implies that there exists $S\in \mathbb{R}$ such that, on a subsequence, $ns_n\rightarrow S$ as $n\rightarrow\infty.$ Therefore, on a subsequence, we have
\begin{equation}\label{eq:2.12}
n\big(u_n- u_*\big)\rightarrow v+S ~\mbox{and}~ u_n\rightarrow u_*~\mbox{in}~W^{1,q}(\Omega)~\mbox{as}~n\rightarrow \infty.
\end{equation}
In addition, there exists $N_2\in \mathbb{N}^*$ such that $u_n\not\equiv 0~\forall~n\geq N_2.$
Now, making use of \eqref{eq:2.5} and \eqref{eq:2.6} it is easy to observe that $u_*$ minimizes functional $J_\mu$ over $ \mathcal{C}_{q}\setminus\{0\}$. By using the minimality of $u_*$ and the fact that $u_n\in \mathcal{C}_{q}\setminus \{0\}~\forall~n\geq N_2,$ we obtain that
\begin{equation}\label{eq:2.13}
0\leq \underset{n\rightarrow\infty}{\lim}  \frac{J_\mu(u_n)-J_\mu(u^*)}{(1/n)}.
\end{equation}
On the other hand,
\begin{equation}\label{eq:2.14}
 n\big({J}_\mu(u_n)-J_\mu(u_*)\big)=\langle {J}^{\prime}_q(u_*), n(u_n-u_*)\rangle+o(n; u_*, v),
\end{equation}
where $o(n; u_*, v)$ is a notation for the  term  which tends to zero in the definition of the Fr\'{e}chet differential of  ${J}_\mu$ at $u_*,$ that is $o(n; u_*, v)\rightarrow 0$ as $n\rightarrow \infty$.
It follows from  \eqref{eq:2.12}-\eqref{eq:2.14} in combination with $u_*\in \mathcal{C}_{1q}$ that
\begin{equation}\label{eq:2.15}
\begin{split}
0\leq&\underset{n\rightarrow \infty}{\lim} n\big(J_\mu(u_n)-J_\mu(u_*)\big)=\underset{n\rightarrow \infty}{\lim}\langle J_\mu^{\prime}(u_*), n(u_n-u_*)\rangle+o(n; u_*, v)\\
&=\langle J_\mu^{\prime}(u_*), v+S\rangle=\langle J_\mu^{\prime}(u_*), v\rangle.
\end{split}
\end{equation}

A similar reasoning with $-v$ instead of $v$ shows that $\langle{J}_\mu (u_*), v\rangle=0$  for every Lipschitz test function, $v$. Taking into account the density of Lipschitz functions in $W^{1,q}(\Omega)$, which is true since $\partial \Omega$ is smooth (hence Lipschitz, see \cite[Theorem 3.6]{G}), we obtain that $u_*$ is an eigenfunction of problem \eqref{eq:2.1} corresponding to eigenvalue $\mu=\lambda_{1q}>0.$

It remains to show that there is no eigenvalue of problem \eqref{eq:2.1} in the open interval $(0, \lambda_{1q}).$

Assume by way of contradiction that there exists $\lambda \in(0, \lambda_{1q})$ for which \eqref{eq:2.1} possesses a solution $u_\lambda\in \mathcal{C}_{q}\setminus\{0\}$. It follows from \eqref{eq:2.2} with $w=u_\lambda$ and \eqref{eq:2.3}  that
\begin{equation*}
\begin{split}
0<&(\lambda_{1q}-\lambda)\Big(\int_\Omega a\mid u_\lambda\mid^q~dx+\int_{\partial\Omega} b\mid u_\lambda\mid^q~d\sigma\Big) \leq \int _\Omega \mid \nabla u_\lambda\mid^q~dx\\
-&\lambda \Big(\int_\Omega a\mid u_\lambda\mid^q~dx+\int_{\partial\Omega} b\mid u_\lambda\mid^q~d\sigma\Big)=0,
\end{split}
\end{equation*}
which is a contradiction. This concludes the proof.
\end{proof}
\begin{remark}\label{remark11}
If $u_\lambda $ is an eigenfunction corresponding to an eigenvalue $\lambda>0$, then we have  from \eqref{eq:1.3}
\[
\int_\Omega \Big(\mid \nabla u_\lambda\mid ^{p}+\mid \nabla u_\lambda\mid ^{q}\Big)~dx=\lambda\Big(\int_\Omega a\mid  u_\lambda\mid ^{q} ~dx+\int_{\partial\Omega} b
\mid  u_\lambda\mid ^{q} ~d\sigma\Big),
\]
thus $u$ cannot be a constant function (see \eqref{eq:1.2}) and so
\[
\int_\Omega a\mid  u_\lambda\mid ^{q} ~dx+\int_{\partial\Omega} b \mid  u_\lambda\mid ^{q} ~d\sigma> 0.
\]
Therefore, denoting $$\Gamma_1(u_\lambda):=\{x\in \Omega;~a(x)u_\lambda(x)\neq 0\},~\Gamma_2(u_\lambda):=\{x\in \partial\Omega;~b(x)u_\lambda(x)\neq 0\},$$ we see that
either $|\Gamma_1(u_\lambda)|_N>0$ or $|\Gamma_2(u_\lambda)|_{N-1}>0$. Obviously, $u_{\lambda}$ corresponding to any eigenvalue $\lambda >0$ cannot be a constant function (see \eqref{eq:1.3} with $v=u_\lambda$
and \eqref{eq:1.2}).
 \end{remark}
\begin{remark}\label{remarca}
Note that the infimum on $\mathcal{C}\setminus\{0\}$ of the Rayleigh-type quotient associated to the eigenvalue problem \eqref{eq:1.1} is given by
\begin{equation}\label{eq:2.16}
\widetilde{\lambda}_1:=\underset{w\in\mathcal{C}\setminus\{0\}}{\inf }~\frac{\frac{1}{q}\int_\Omega\mid\nabla w\mid^q~dx+\frac{1}{p}\int_\Omega\mid\nabla w\mid^p~dx}{\frac{1}{q}\big(\int_{\Omega}a\mid w\mid^q~dx+\int_{\partial\Omega}b\mid w\mid^q~d\sigma\big)}.
\end{equation}
In fact, $\widetilde{\lambda}_1=\lambda_1.$ Indeed, it is obvious that $\lambda_1\leq\widetilde{\lambda}_1$ and for the converse inequality we
note that, $\forall v\in
\mathcal{C}\setminus\{0\}$, $t>0,$ we have $tv\in \mathcal{C}\setminus\{0\}$ and
\begin{equation*}
\begin{split}
\widetilde{\lambda}_1&=\underset{w\in\mathcal{C}\setminus\{0\}}{\inf }~\frac{\frac{1}{q}\int_\Omega\mid\nabla w\mid^q~dx+
\frac{1}{p}\int_\Omega\mid\nabla w\mid^p~dx}{\frac{1}{q}\big(\int_{\Omega}a\mid w\mid^q~dx+\int_{\partial\Omega}b\mid w\mid^q~d\sigma\big)}\leq\\
&\frac{\int_\Omega\mid\nabla v\mid^q~dx}{\int_{\Omega}a\mid v\mid^q~dx+\int_{\partial\Omega}b\mid v\mid^q~d\sigma}+
t^{p-q}\frac{q\int_\Omega\mid\nabla v\mid^p~dx}{p\big(\int_{\Omega}a\mid w\mid^q~dx+\int_{\partial\Omega}b\mid v\mid^q~d\sigma\big)}.
\end{split}
\end{equation*}
Now letting $t\rightarrow\infty$ if $q>p$, and $t\rightarrow 0$ if $q<p$, then passing to infimum for $v\in \mathcal{C}\setminus\{0\}$
we get the desired inequality. Hence $\lambda_1$ can be expressed in two different ways (see \eqref{eq:1.6} and \eqref{eq:2.16}).
\end{remark}

\begin{remark}\label{remark2}
As a consequence of Lemma~\ref{lema1} we have $\lambda_1>0$. Indeed, from \eqref{eq:1.6} we have
\[
\lambda_1:=\underset{w\in\mathcal{C}_1}{\inf }~\int_{\Omega}\mid\nabla w\mid^q~dx \, ,
\]
where $\mathcal{C}_1=\{v\in \mathcal{C}; \int_{\Omega}a\mid v\mid^q~dx+\int_{\partial\Omega}b\mid v\mid^q~d\sigma=1\}$. So $\lambda_1=J(u^*)$ for $p\leq q$
  and $\lambda_1\geq J(u^*)$ if $p > q.$ Thus in both cases $\lambda_1>0.$
\end{remark}

\section{ Proof of the main result}

We have already stated that $\lambda_0=0$ is an eigenvalue of problem \eqref{eq:1.1} and any other eigenvalue of this problem belongs to $(0,\infty)$.
We verify next that no eigenvalue belongs to $(0, \lambda_1].$ To argue by contradiction, assume that problem \eqref{eq:1.1} possesses an eigenvalue $\lambda\in (0, \lambda_1]$ with a corresponding eigenfunction $u_\lambda.$ Then, from \eqref{eq:1.3}
\begin{equation}\label{eq:3.1}
\int_\Omega \big(\mid \nabla u_\lambda\mid ^{p}+\mid \nabla u_\lambda\mid ^{q}\big)~dx=\lambda\Big(
\int_\Omega a\mid  u_\lambda\mid ^{q}~dx+\int_{\partial\Omega} b \mid  u_\lambda\mid ^{q} ~d\sigma
\Big).
\end{equation}
Note that $\int_\Omega a\mid  u_\lambda\mid ^{q}~dx+\int_{\partial\Omega} b \mid  u_\lambda\mid ^{q} ~d\sigma \neq 0$, otherwise $u_{\lambda}\equiv Const.$
which is impossible (see Remark \ref{remark11}).
On the other hand, as $u_\lambda\in \mathcal{C}\setminus\{0\}$, we derive from \eqref{eq:1.6} and \eqref{eq:3.1}
\begin{equation*}
\begin{split}
\lambda\leq \lambda_1\leq &\frac{\int_\Omega\mid\nabla u_\lambda\mid^q~dx}{\int_{\Omega}a\mid
 u_\lambda \mid^q~dx+\int_{\partial\Omega}b\mid u_\lambda\mid^q~d\sigma}\\
&=\frac{\lambda\Big(\int_{\Omega}a\mid u_\lambda \mid^q~dx+\int_{\partial\Omega}b\mid u_\lambda\mid^q~d\sigma\Big)-
\int_\Omega\mid \nabla u_\lambda\mid^p~dx}{\int_{\Omega}a\mid u_\lambda \mid^q~dx+\int_{\partial\Omega}b\mid u_\lambda\mid^q~d\sigma}\\
&= \lambda-\frac{\int_\Omega\mid\nabla u_\lambda\mid^p~dx}{\int_{\Omega}a\mid u_\lambda \mid^q~dx+
\int_{\partial\Omega}b\mid u_\lambda\mid^q~d\sigma}< \lambda,
\end{split}
\end{equation*}
which is a contradiction.

In what follows we shall prove that every $\lambda>\lambda_1$ is an eigenvalue of problem \eqref{eq:1.1}. To this purpose we fix such a $\lambda$ and define $\mathcal{J}_\lambda: W\rightarrow \mathbb{R},$
\begin{equation}\label{eq:3.2}
\mathcal{J}_\lambda(u)=\frac{1}{p}\int_\Omega \mid \nabla u\mid ^{p}~dx+\frac{1}{q}\int_\Omega\mid \nabla u\mid ^{q}~dx-
\frac{\lambda}{q}\Big(\int_\Omega a\mid  u\mid ^{q}~dx+\int_{\partial\Omega} b \mid  u_\lambda\mid ^{q} ~d\sigma\Big),
\end{equation}
which is a $C^1$ functional whose derivative is given by
\begin{equation}\label{eq:3.3}
\begin{split}
\langle\mathcal{J}'_\lambda(u),w\rangle&=\int_\Omega \mid \nabla u\mid ^{p-2}\nabla u\cdot\nabla w~dx+
\int_\Omega\mid \nabla u\mid ^{q-2}\nabla u\cdot\nabla w~dx\\
&-\lambda\Big(\int_\Omega a\mid  u\mid ^{q-2}uw~dx+\int_{\partial\Omega} b \mid  u_\lambda\mid ^{q-2}uw ~d\sigma\Big)\ \
\forall u, w\in W.
\end{split}
\end{equation}
So, according to Definition~\ref{def1}, $\lambda>\lambda_1$ is an eigenvalue of problem \eqref{eq:1.1} if and only if there exists
a critical point $u_\lambda\in W\setminus\{0\}$ of $\mathcal{J}_\lambda$, i. e.  $\mathcal{J}'_\lambda(u_\lambda)=0$.

The proof of Theorem~\ref{teorema1} will follow as a consequence of several intermediate results.
We shall discuss two cases which are complementary to each other.

\textbf{Case 1: $q\in (1, \infty),~p>q$}.

In this case we have $W=W^{1,p}(\Omega).$ The following lemma shows, essentially, that the functional defined in \eqref{eq:3.1}
is coercive for every $\lambda >\lambda_1$ restricted to the subset $\mathcal{C}\subset W=W^{1,p}(\Omega).$

\begin{lemma}\label{lema3.1}
 Let $q\in (1, \infty),~p>q.$ For every $\lambda >\lambda_1$, we have
\[
\underset{\parallel u\parallel_{W^{1,p}(\Omega)}\rightarrow\infty, u\in\mathcal{C}}{\lim}\mathcal{J}_\lambda(u)=\infty.
\]
\end{lemma}
For the proof of this lemma we refer the reader to L. Barbu and G. Moro\c{s}anu \cite[Case 1]{BM}.

\begin{lemma}\label{lemma3.2}
 Let $q\in (1, \infty),~p>q.$ Every number $\lambda\in (\lambda_1, \infty)$ is an eigenvalue of
problem \eqref{eq:1.1}.
\end{lemma}
\begin{proof}
Note that $\mathcal{C}$ is a weakly closed subset of the reflexive Banach space $W=W^{1,p}(\Omega),$ and functional $\mathcal{J}_\lambda$ is coercive (see Lemma~\ref{lema3.1}) and weakly lower semicontinuous on $\mathcal{C}$ with respect to the norm of $W^{1,p}(\Omega).$
Standard results in the calculus of variations (see, e.g., \cite [Theorem 1.2] {St}) ensures the existence of a global minimizer $z_{*}\in \mathcal{C}$ for $\mathcal{J}_\lambda$, i.e., $\mathcal{J}_\lambda(z_{*})=\min_{\mathcal{C}}\mathcal{J}_\lambda$.

From Remark \ref{remarca} we know that $\lambda_1=\widetilde{\lambda}_1$, hence
 $\lambda>\lambda_1=\widetilde{\lambda}_1$. Then (by \eqref{eq:2.16}) there exists $u_{0\lambda}\in \mathcal{C}\setminus\{0\}$ such that
$\mathcal{J}_\lambda(u_{0\lambda})<0.$ It follows that
\[
\mathcal{J}_\lambda(z_{*})\leq \mathcal{J}_\lambda(u_{0\lambda})<0,
\]
which shows that $z_{*}\neq 0.$

Next, we are going to show that the global minimizer $z_*$ for $\mathcal{J}_\lambda$ restricted to $\mathcal{C}$ is a critical point of $\mathcal{J}_\lambda$ considered on the whole space $W^{1,p}(\Omega),$ i. e., $\mathcal{J}_\lambda^\prime(z_*)=0,$ in other words, $z_*$ is an eigenfunction of problem \eqref{eq:1.1} corresponding to $\lambda.$

In fact, $z_{*}$ is a solution of the minimization problem
\[
\min_{w\in W} \mathcal{J}_\lambda(w),
\]
under the restriction
\[
g(w):=\int_\Omega a\mid  w\mid ^{q-2} w  ~dx+\int_{\partial\Omega} b \mid  w\mid ^{q-2} w  ~d\sigma=0.
\]
If $q\in [2, \infty),~p>q,$ we have proved in \cite[Case 1]{BM}, by using the Lagrange Multipliers Rule, that $\mathcal{J}_\lambda(z_*)=0$. For $q \in (1,2)$, $g$ is no longer  a $C^1$ function on $W$, so we cannot use  the same reasoning  to prove our assertion. Fortunately, we can use a technique similar to that used in the proof of Lemma~\ref{lema1}. It is worth mentioning that this technique works for the case $q\in [2, \infty ),$ too.

Since $p>q$, the inclusions $W^{1,p}(\Omega)\hookrightarrow L^q(\Omega)$ and  $W^{1,p}(\Omega)\hookrightarrow L^q(\partial\Omega)$ are compact. As in the proof of Lemma~\ref{lema1}, let us fix an arbitrary $v\in \mbox{Lip}(\Omega)$ and construct the sequence
\begin{equation}\label{eq:3.4}
u_n:=z_*+\frac{1}{n}v+s_n~\forall~n\in \mathbb{N}^*,
\end{equation}
such that $\big(u_n\big)_n \subset \mathcal{C}.$

Similar arguments as in the proof of Lemma~\ref{lema1} can be used in order to prove that the sequence $\big(n s_n\big)_n$ is also bounded, hence it converges on a subsequence to  some $S\in \mathbb{R}$ and so, on a subsequence,
\begin{equation}\label{eq:3.5}
n\big(u_n- z_*\big)\rightarrow v+S ~\mbox{and}~ u_n\rightarrow z_*~\mbox{in}~W^{1,p}(\Omega)~\mbox{as}~n\rightarrow \infty.
\end{equation}
Since $z_{*}$ minimizes functional $\mathcal{J}_\lambda$ over $ \mathcal{C}$ and $\big(u_{\lambda n}\big)_n \subset \mathcal{C}$, we have
\begin{equation}\label{eq:3.7}
0\leq \underset{n\rightarrow\infty}{\lim}  \frac{\mathcal{J}_\lambda(u_{ n})-\mathcal{J}_\lambda(z_{*})}{\frac{1}{n}}.
\end{equation}
We also have
\begin{equation}\label{eq:3.8}
 n\big(\mathcal{J}_\lambda(u_{ n})-\mathcal{J}_\lambda(z_{*})\big)=\langle \mathcal{{J}}^{\prime}_\lambda(z_{*}), n(u_{n}-z_{*})\rangle+o(n; z_{\lambda*}, v),
\end{equation}
with $o(n; z_{*}, v)\rightarrow 0$ as $n\rightarrow \infty$.
From  \eqref{eq:3.5}-\eqref{eq:3.8}, combined with $z_{*}\in \mathcal{C},$  we get
\begin{equation}\label{eq:3.9}
\begin{split}
0\leq&\underset{n\rightarrow \infty}{\lim} n\big(\mathcal{J}_\lambda(u_{n})-\mathcal{J}_\lambda(z_{*})\big)=\underset{n\rightarrow \infty}{\lim}\langle \mathcal{J}_\lambda^{\prime}(z_{*}), n(u_{n}-z_{*})\rangle+o(n; z_{*}, v)\\
&=\langle \mathcal{J}_\lambda^{\prime}(z_{*}), v+S\rangle=\langle \mathcal{J}_\lambda^{\prime}(z_{*}), v\rangle.
\end{split}
\end{equation}
A similar reasoning with $-v$ instead of $v$ and the density of Lipschitz functions in $W^{1,p}(\Omega)$ yield  $\mathcal{J}_\lambda^\prime (z_{*})=0$, which concludes the proof.
\end{proof}

\textbf{Case 2: $q\in (1,\infty)$, $p<q$}.

In this case $W=W^{1,q}(\Omega)$ and $\mathcal{C}=\mathcal{C}_q.$ Let $\lambda>\lambda_1$ be a fixed number. Under the assumption $p<q$ we cannot expect coercivity on $W^{1,q}(\Omega)$ of the functional $\mathcal{J}_\lambda$.  From now on we analyse the action of $\mathcal{J}_\lambda$ on the Nehari type manifold (see \cite{SW}) defined by
$$
\mathcal{N}_\lambda=\{v\in \mathcal{C}\setminus\{0\}; \langle \mathcal{J}'_\lambda(v),v\rangle=0\}
$$
$$
=\Big\{v\in \mathcal{C}\setminus\{0\}; \int_\Omega \big(\mid \nabla v\mid ^{p}+\mid \nabla v\mid ^{q}\big)~dx
=\lambda\Big(\int_\Omega a\mid  v\mid ^{q} ~dx+\int_{\partial\Omega} b \mid  v\mid ^{q} ~d\sigma\Big)\Big\}.
$$
It is natural to consider the restriction of $\mathcal{J}_\lambda$ to $\mathcal{N}_\lambda$ since any possible eigenfunction corresponding to $\lambda$ belongs to $\mathcal{N}_\lambda$. Note that
on $\mathcal{N}_\lambda$ functional $\mathcal{J}_\lambda$ has the form
\begin{equation*}
\begin{split}
\mathcal{J}_\lambda(u)=&\frac{1}{p}\int_\Omega \mid \nabla u\mid ^{p}~dx+\frac{1}{q}\int_\Omega\mid \nabla u\mid ^{q}~dx-\frac{\lambda}{q}\Big(\int_\Omega a\mid  u\mid ^{q}~dx+\int_{\partial\Omega} b \mid  u\mid ^{q} ~d\sigma\Big)\\
=&\frac{1}{p}\int_\Omega \mid \nabla u\mid ^{p}~dx-\frac{1}{q}\int_\Omega\mid \nabla u\mid ^{p}~dx=\frac{q-p}{qp}\int_\Omega \mid \nabla u\mid ^{p}~dx>0.
\end{split}
\end{equation*}
Now, let us recall the following result from L. Barbu and G. Moro\d{s}anu \cite[Case 2, Steps 1-4]{BM}.
 \begin{lemma}\label{lemma3.3}
 Let $q\in (1, \infty),~p<q.$ Then there exists a point $u_{*}\in \mathcal{N}_\lambda$ where $\mathcal{J}_\lambda$ attains its minimal value, $m_\lambda:= \underset{w\in\mathcal{{N}}_\lambda}{\inf }{\mathcal{{J}}_\lambda (w) }>0.$
 \end{lemma}

In the sequel we show that the minimizer $u_*$, given by Lemma~\ref{lemma3.3}, is a critical point of $\mathcal{J}_\lambda$ considered on the whole space $W^{1,q}(\Omega)$.
 \begin{lemma}\label{lemma3.4}
 Let $q \in (1,\infty),~p<q.$ The minimizer $u_{*}\in \mathcal{N}_\lambda$ from Lemma~\ref{lemma3.3} is an eigenfunction of problem \eqref{eq:1.1} with corresponding eigenvalue $\lambda.$
 \end{lemma}

\begin{proof} It suffices to prove that $\mathcal{J}'_\lambda( u_*)=0.$

In fact $u_*$  is a minimizer of $J_{\lambda}$ for $w\in W$ subject to
the restrictions
\begin{equation}\label{eq:3.10}
g_1(w):=\int_\Omega \big(\mid \nabla w\mid ^{p}+\mid \nabla w\mid ^{q}\big)~dx
-\lambda\Big(\int_\Omega a\mid  w\mid ^{q} ~dx+\int_{\partial\Omega} b \mid w\mid ^{q} ~d\sigma\Big)=0,
\end{equation}
\begin{equation}\label{eq:3.11}
g_2(w):=\int_\Omega a\mid  w\mid ^{q-2} w  ~dx+\int_{\partial\Omega} b \mid  w\mid ^{q-2} w  ~d\sigma=0.
\end{equation}

In the case $q\in [2, \infty),~ p<q,$ the conclusion was proved in L. Barbu and G. Morosanu \cite[Step 5]{BM}, by using the Lagrange Multipliers Rule. If $q\in (1,2)$ , the function  $g_2$ is not in $C^1(W;\mathbb{R})$, so the Lagrange Multipliers Rule is no longer applicable to this case. What we can do is to apply a reasoning similar to that used in the proofs of Lemmas~ \ref{lema1} and \ref{lemma3.2} to show that $\mathcal{J}'_\lambda( u_*)=0.$

So, let $v\in \mbox{Lip}(\Omega)$  be an arbitrary but fixed function. Let $u_*\in \mathcal{N}_\lambda $ be the minimizer of $\mathcal{J}_\lambda$ over $\mathcal{N}_\lambda$, and consider the sequence $\big(u_n\big)_n \subset W^{1,q}(\Omega),$
\begin{equation}\label{eq:3.12}
u_{n}:=u_{*}+\frac{1}{n}v+s_{n}~\forall~n\in \mathbb{N}^*,
\end{equation}
with $\big(u_{n}\big)_n \subset \mathcal{C}_q$. Again, the sequence $\big( n s_{n}\big)_n$ is bounded, so it converges on a subsequence to some $S\in \mathbb{R}$. Therefore, on a subsequence, we have
\begin{equation}\label{eq:3.13}
n\big(u_{n}- u_{ *}\big)\rightarrow v+S,~u_{n}\rightarrow u_*~\mbox{in}~W^{1,q}(\Omega)~\mbox{as}~n\rightarrow \infty.
\end{equation}
Since $u_*\not\equiv 0,$ one can assume that $\big(u_n\big)_n \subset \mathcal{C}_q \setminus \{0\}.$ Using this last subsequence of $\big(u_n\big)_n$, we shall construct a sequence $\big(t_n\big)_n \subset \mathbb{R}$ such that $\big(t_n u_n \big)_n\subset \mathcal{N}_\lambda,$ for every $n$ sufficiently large, i.e.,
\begin{equation}\label{eq:3.14}
\begin{split}
t_n^p\int_\Omega \mid \nabla u_n\mid ^{p}~dx+t_n^q\int_\Omega\mid \nabla u_n\mid ^{q}~dx
=\lambda t_n^q\Big(\int_\Omega a\mid  u_n\mid ^{q} ~dx+\int_{\partial\Omega} b \mid  u_n\mid ^{q} ~d\sigma\Big),
\end{split}
\end{equation}
or, equivalently,
\begin{equation}\label{eq:3.15}
t_n=\Biggl(\frac{\int_\Omega \mid \nabla u_n\mid ^{p}~dx }{\lambda \big(\int_{\Omega}a\mid u_n\mid^q~dx+\int_{\partial\Omega}b\mid u_n\mid^q~d\sigma\big)-\int_\Omega\mid\nabla u_n\mid^q~dx}\Biggr)^{1/(q-p)}.
\end{equation}
Note that for sufficiently large $n$, both the numerator and the denominator are positive numbers. Indeed, since  $u_*\in \mathcal{N}_\lambda,$ we have
\begin{equation}\label{eq:3.16}
\int_\Omega \mid \nabla u_*\mid ^{p}~dx>0~~\mbox{and}~~ \int_\Omega\mid \nabla u_*\mid ^{q}~dx<\lambda\Big(\int_\Omega a\mid  u_*\mid ^{q} ~dx+\int_{\partial\Omega} b \mid  u_*\mid ^{q} ~d\sigma\Big).
\end{equation}
Since the functionals
\begin{equation}\label{eq:3.17}
\begin{split}
\mathcal{I}_1, \mathcal{I}_2:&W\rightarrow\mathbb{R},~~~\mathcal{I}_1(w):=\int_\Omega \mid \nabla w\mid ^{p}~dx,\\
\mathcal{I}_2(w)&:=-\int_{\Omega}\mid \nabla w\mid ^{q}~dx+\lambda\Big(\int_\Omega a\mid  w\mid ^{q} ~dx+\int_{\partial\Omega} b \mid  w\mid ^{q} ~d\sigma\Big)~\forall~w\in W
\end{split}
\end{equation}
are continuous on $W$ and $\mathcal{I}_1(u_*)>0,~\mathcal{I}_2(u_*)>0,$ (see \eqref{eq:3.16}), there exists $\delta_0>0$ such that
$$
w\in W,~\parallel w-u_*\parallel_W<\delta_0 \ \Longrightarrow \mathcal{I}_1(w)>0, \ \mathcal{I}_2(w)>0.
$$
Since $u_n\rightarrow u_*$ in $W,$ it follows that for $N_0$  large enough, $\mathcal{I}_1(u_n)>0, \mathcal{I}_2(u_n)>0~\forall~n\geq N_0, $ hence $t_n$ given by \eqref{eq:3.15} is well
defined for $n\ge N_0$. So we can define
\begin{equation}\label{eq:3.18}
z_{n}:=t_n\Biggl(u_{*}+\frac{1}{n}v+s_{n}\Biggr)=t_n u_n~\forall~n\geq N_0,
\end{equation}
with $\big(z_{n}\big)_n \subset \mathcal{N}_\lambda.$
In addition, using \eqref{eq:3.15} and \eqref{eq:3.18}, we can see that
\begin{equation}\label{eq:3.19}
t_n\rightarrow 1~\mbox{in}~\mathbb{R}, ~z_{ n}\rightarrow u_*~\mbox{in}~W^{1,q}(\Omega)~ \mbox{as}~n\rightarrow \infty.
\end{equation}
In what follows we shall prove that the sequence  $\big(n(t_n-1)\big)_n$ is bounded. To this purpose, let us first show that  the sequence $\big(n(t_n^{p-q}-1)\big)_n$ is bounded. Define the functional $\mathcal{L}_\lambda:W\rightarrow\mathbb{R},$
\begin{equation}\label{eq:3.20}
\begin{split}
\mathcal{L}_\lambda(u)=&-\int_\Omega \mid \nabla u\mid ^{p}~dx-\int_\Omega\mid \nabla u\mid ^{q}~dx\\
&+\lambda\Big(\int_\Omega a\mid  u\mid ^{q}~dx+\int_{\partial\Omega} b \mid  u\mid ^{q} ~d\sigma\Big)~\forall~u \in W,
\end{split}
\end{equation}
which belongs to $C^1(W;\mathbb{R})$, and for  $u, w \in W$
\begin{equation}\label{eq:3.21}
\begin{split}
\langle\mathcal{L}'_\lambda(u),w\rangle&=-p\int_\Omega \mid \nabla u\mid ^{p-2}\nabla u\cdot\nabla w~dx-q
\int_\Omega\mid \nabla u\mid ^{q-2}\nabla u\cdot\nabla w~dx\\
&+\lambda q\Big(\int_\Omega a\mid  u\mid ^{q-2}uw~dx+\int_{\partial\Omega} b \mid  u_\lambda\mid ^{q-2}uw ~d\sigma\Big).
\end{split}
\end{equation}
From \eqref{eq:3.20} and $u_*\in \mathcal{N}_\lambda$, we infer that $\mathcal{L}_\lambda (u_*)=0,$ so we get
\begin{equation}\label{eq:3.22}
n(t_n^{p-q}-1)=\frac{n\big(\mathcal{L}_\lambda (u_n)-\mathcal{L}_\lambda (u_*)\big)}{\int_\Omega \mid \nabla u_n\mid ^{p}~dx}.
\end{equation}
Since $p<q$, we have
\begin{equation}\label{eq:3.23}
\int_\Omega \mid \nabla u_n\mid ^{p}~dx\rightarrow\int_\Omega \mid \nabla u_*\mid ^{p}~dx> 0,
\end{equation}
\begin{equation}\label{eq:3.24}
n\big(\mathcal{L}_\lambda (u_n)-\mathcal{L}_\lambda (u_*)\big)\rightarrow \langle \mathcal{L}'_\lambda(u_*), v+S\rangle~ \mbox{as}~n\rightarrow \infty.
\end{equation}
From \eqref{eq:3.22} and \eqref{eq:3.24} we deduce that the sequence $\big(n(t_n^{p-q}-1)\big)_n$ has a finite limit. Hence, there is $K>0$ such that for all $n\geq N_0,$ $n\mid t_n^{p-q}-1\mid \leq K,$ which implies
\[
1-\frac{K}{n}\leq t_n^{p-q}\leq 1+\frac{K}{n}~\forall ~n\geq N_0.
\]
Since, there exists $N_1\in \mathbb{N}^*$ such that $1-K/n >0~\forall~n\geq N_1,$ we have
\begin{equation}\label{eq:3.25}
n\Biggl(\Big(1+\frac{K}{n}\Big)^{\frac{1}{p-q}}-1\Biggr)\leq n(t_n-1)\leq n\Biggl(\Big(1-\frac{K}{n}\Big)^{\frac{1}{p-q}}-1\Biggr)~\forall ~n\geq \max \{N_0, N_1\}.
\end{equation}
Taking into account the relations
\[
\underset{x\rightarrow 0}{\lim}\frac{(1+Kx)^{1/(p-q)}-1}{x}=K/(p-q),~\underset{x\rightarrow 0}{\lim}\frac{(1-Kx)^{1/(p-q)}-1}{x}=-K/(p-q),~
\]
we infer from \eqref{eq:3.25} that the sequence $\big(n(t_n-1)\big)_n$ is bounded, thus, by possibly passing to a subsequence, there exists $T\in \mathbb{R}$, such that $n(t_n-1)\rightarrow T$ as $n\rightarrow \infty$.

 By using the minimality of $u_*$ and the fact that $\big(z_n\big)_n\subset \mathcal{N}_\lambda$ we obtain that
\begin{equation}\label{eq:3.26}
0\leq \underset{n\rightarrow\infty}{\lim}  \frac{\mathcal{J}_\lambda(z_n)-\mathcal{J}_\lambda(u_*)}{\frac{1}{n}}.
\end{equation}
Since functional $\mathcal{J}_\lambda\in C^1(W;\mathbb{R}),$  we can write
\begin{equation}\label{eq:3.27}
n\big(\mathcal{J}_\lambda(z_n)-\mathcal{J}_\lambda(u_*)\big)=\big(\langle \mathcal{J}_\lambda^{\prime}(u_*), n(z_n-u_*)\rangle+o(n; u_*, v),
\end{equation}
with $o(n; u_*, v)\rightarrow 0$ as $n\rightarrow \infty$. Taking into account   \eqref{eq:3.18} and  \eqref{eq:3.19}, we can see that, on a subsequence,
\begin{equation}\label{eq:3.28}
n(z_n-u_*)=n \big(t_n-1\big)u_*+v+ns_n\rightarrow T u_* +v+S~~\mbox{as}~n\rightarrow \infty~\mbox{in}~W.
\end{equation}
It follows from \eqref{eq:3.26} and \eqref{eq:3.28} that
\begin{equation}\label{eq:3.29}
0\leq \langle \mathcal{J}_\lambda^{\prime}(u_*), v+S+Tu_*\rangle.
\end{equation}
Since  $u_*\in \mathcal{N}_\lambda,$ we obtain that $\langle \mathcal{J}_\lambda^{\prime}(u_*), u_*\rangle=0,~\langle \mathcal{J}_\lambda^{\prime}(u_*), S\rangle=0,$ hence \eqref{eq:3.29} implies
\[
0\leq \langle J_\mu^{\prime}(u_*), v\rangle.
\]
A similar reasoning with $-v$ instead of $v$ shows that the converse inequality holds, hence
$0=\langle \mathcal{J}_\lambda^{\prime}(u_*), v\rangle$.
Finally, using the density of Lipschitz functions in $W$ we obtain that $\mathcal{J}_\lambda^{\prime}(u_*)=0,$ which concludes the proof.
\end{proof}
 Therefore, as it has already been pointed out, $\lambda=0$  is an eigenvalue, so the conclusion of Theorem~\ref{teorema1} follows from Lemma~\ref{lemma3.2} and Lemma~\ref{lemma3.4}.

\begin{remark}
Thus, if $q>1$ and $1<p<q$ then $\lambda_1 = \lambda_{1q}$, so the eigenvalue set of problem \eqref{eq:1.1} is $\{ 0\}\cup (\lambda_{1q}, \infty)$,
which is independent of $p$. If $1<q<p$ then $\lambda_1\ge \lambda_{1q}$.
\end{remark}


\begin{thebibliography}{}


\bibitem{AM} Abreu, J.,  Madeira, G., Generalized eigenvalues of the $(P, 2)-$Laplacian under a parametric boundary condition, Proc. Edinburgh Math. Soc., \textbf{63}(1) (2020), 287-303.

\bibitem{A}  Anderson D., Jancel R., Wilhelmsson, H., Phys. Rev. A 30 (1984), \textbf{2}, 965–
966.

\bibitem{Ar} Aris, R., Mathematical modelling techniques, Research Notes in Mathematics,
\textbf{24}, Pitman (Advanced Publishing Program), Boston, Mass.-London, 1979.

\bibitem{BM} Barbu, L., Moro\c{s}anu, G., Eigenvalues of the negative (p,q)- Laplacian under a Steklov-like boundary condition, Complex Var. Elliptic Equations, \textbf{64}(4) (2019), 685–700.

\bibitem{BF}  Brasco, L., Franzina, G. An anisotropic eigenvalue problem of Stekloff type and weighted Wulff inequalities, Nonlinear Differ. Equ. Appl. \textbf{20} (2013), 1795-1830.

\bibitem{Br} Brezis, H., {Functional Analysis, Sobolev Spaces and Partial Differential
Equations}, Springer, 2011.

\bibitem{CF} Casas, E., Fern\'{a}ndez, L.A.,  {A Green's formula for quasilinear elliptic operators}, {J. Math. Anal. Appl.}, \textbf{142}(1989), 62-73.


\bibitem{FMS} F\u{a}rc\u{a}\c{s}eanu, M., Mih\u{a}ilescu M., Stancu-Dumitru, D., {On the set of eigen-
values of some PDEs with homogeneous Neumann boundary condition}, {Nonlinear Anal. Theory Methods Appl.},  \textbf{116} (2015), 19-25.

\bibitem{F} Fife, P.C., Mathematical aspects of reacting and diffusing systems, Lecture
Notes in Biomathematics, 28, Springer-Verlag, Berlin-New York, 1979.


\bibitem{Fo} Folland, G.B., {Real Analysis: Modern Techniques and Their Applications (2nd ed.)}, Pure and Applied Mathematics,  John Wiley $\&$ Sons, Inc., New York, 1999.

\bibitem{GP} Gasinski, L., Papageorgiou, N.S., Nonlinear Analysis, Series in Mathematical Analysis and Applications, 9, Chapman \& Hall/CRC, Boca Raton, FL, 2006.


\bibitem{G} Giga, Y., {Surface Evolution Equations. A Level Set Approach}, Birkh\"{a}user Verlag:, Basel, 2006.


\bibitem{MMih} Mih\u{a}ilescu, M., {An eigenvalue problem possesing a continuous family of eigenvalues plus an isolated eigenvale}, {Commun. Pure
Appl. Anal.} \textbf{10} (2011), 701-708.

\bibitem{MM} Mih\u{a}ilescu,  M., Moro\c{s}anu, G., {Eigenvalues of $-\triangle_p-\triangle_q$ under Neumann
boundary condition}, {Canadian Math. Bull.}, 59(3) (2016), 606-616.

\bibitem{Mu} Murray, J.D., Mathematical biology, Biomathematics, \textbf{19,} Springer-Verlag, Berlin, 1993.

\bibitem{My} Myers-Beaghton, A.K., Vvedensky, D. D., Chapman-Kolmogorov equation for
Markov models of epitaxial growth. J. Phys. A, 22(11) (1989), 467 - 475.

\bibitem{St}  Struwe, M., {Variational Methods: Applications to Nonlinear Partial Differential Equations and Hamiltonian Systems}, Springer, 1996.

\bibitem{SW} Szulkin, A., Weth, T., {The Method of Nehari Manifold, Handbook of
Nonconvex Analysis and Applications}, Int. Press, Somerville, MA, 597-632, 2010.

\bibitem{W} Wilhelmsson, H., Explosive instabilities of reaction-diffusion equations, Phys.
Rev. A (3) 36 (1987), no. 2, 965–966.

\end{thebibliography}
\end{document}